\documentclass[bj,numbers,noshowframe]{imsart}% uncomment this for numbers citation
\usepackage{bbm}
\usepackage{csquotes}
\usepackage{mathtools}

\startlocaldefs
\numberwithin{equation}{section}
\theoremstyle{plain}
\newtheorem{theorem}{Theorem}[section]
\newtheorem{lemma}[theorem]{Lemma}

\newtheorem{proposition}[theorem]{Proposition}

\newtheorem{corollary}[theorem]{Corollary}
\theoremstyle{definition}
\newtheorem{definition}[theorem]{Definition}
\newtheorem{assumption}{Assumption}
\newtheorem{remark}[theorem]{Remark}
\newcommand{\citewithin}[2]{\protect{\cite[#1]{#2}}}
\newcommand{\condref}[2]{\textnormal{(\textcolor{blue}{\hyperref[#1]{#2}})}}
\newcommand{\klr}[1]{\left(#1\right)}
\newcommand{\kle}[1]{\left[#1\right]}
\newcommand{\klg}[1]{\left\{#1\right\}}
\newcommand{\kls}[1]{\left|#1\right|}
\newcommand\nat[0]{\mathbb{N}}
\newcommand{\norm}[1]{\left \lVert #1 \right \rVert}
\endlocaldefs

\begin{document}

\begin{frontmatter}
\title{Strong solutions for singular SDEs driven by long-range dependent fractional Brownian motion and other Volterra processes}
%\title{A sample article title with some additional note\thanksref{T1}}
\runtitle{Strong Solutions for Singular SDEs with Long-Range Dependent fBM}
%\thankstext{T1}{A sample of additional note to the title.}

\begin{aug}
%%%%%%%%%%%%%%%%%%%%%%%%%%%%%%%%%%%%%%%%%%%%%%%
%% Additional information such as            %%
%% identifying the corresponding author must %%
%% be included in in the Acknowledgments     %%
%% section if necessary.                     %%
\author[A,B]{\inits{M.}\fnms{Maximilian}~\snm{Buthenhoff}\ead[label=e1]{buthenhoff.m.b890@m.isct.ac.jp}\orcid{0009-0007-7599-6251}}
\author[C]{\inits{E.}\fnms{Ercan}~\snm{S\"onmez}\ead[label=e2]{soenmez@postech.ac.kr}}
%%%%%%%%%%%%%%%%%%%%%%%%%%%%%%%%%%%%%%%%%%%%%%
%% Addresses                                %%
%%%%%%%%%%%%%%%%%%%%%%%%%%%%%%%%%%%%%%%%%%%%%%
\address[A]{Institute for Theoretical
Physics III, Ruhr University Bochum, Bochum, Germany\printead[presep={,\ }]{}}

\address[B]{Department of Physics, Institute of Science Tokyo, Tokyo, Japan\printead[presep={,\ }]{e1}}

\address[C]{Department of Mathematics, Pohang University of Science and Technology, 77 Cheongam-ro, Nam-gu,
Pohang-si, Gyeongsangbuk-do, Republic of Korea\printead[presep={,\ }]{e2}}
\end{aug}

\begin{abstract}
We investigate the well-posedness of stochastic differential equations driven by fractional Brownian motion, focusing on the long-range dependent case $H \in (\frac12, 1)$. While existing results on regularization by such noise typically require Hölder continuity of the drift, we establish new strong existence and uniqueness results for certain classes of singular drifts, including discontinuous and highly irregular functions. More generally, we treat stochastic differential equations with additive noise given by a broader class of Volterra processes satisfying suitable kernel conditions, which, in addition to fractional Brownian motion, also includes the Riemann–Liouville process as a special case. Our approach relies on probabilistic arguments.
\end{abstract}

\begin{keyword}
\kwd{Fractional Brownian motion}
\kwd{Long-range dependence}
\kwd{Stochastic differential equations}
\kwd{Volterra processes}
\end{keyword}

\end{frontmatter}

%%%%%%%%%%%%%%%%%%%%%%%%%%%%%%%%%%%%%%%%%%%%%%
%%%% Main text entry area:
\section{Introduction}
Stochastic differential equations (SDEs) driven by fractional Brownian motion (fBm) have received significant attention due to both their theoretical significance and their applications in mathematical finance, physics and stochastic analysis. In particular, the theory of regularization by noise has been a highly active area of research: in certain cases, the addition of a stochastic perturbation of the form
\[
    \mathrm{d}X_t = b(t,X_t) \mathrm{d}t + \mathrm{d} B_t^H,
\]
where \( B_t^H \) is an fBm with Hurst parameter \( H\in (0,1) \), can restore well-posedness to otherwise ill-posed deterministic equations. This idea has been extensively studied when the noise is given by a standard Brownian motion or an fBm with Hurst parameter \( H < 1/2 \), where pathwise techniques play a crucial role. However, the case \( H > 1/2 \) with discontinuous drift function remains far less understood, despite its relevance for rough-volatility modeling in finance \cite{willinger1999stock,hiraki2025survey} and to long-range dependent systems in physics \cite{kou2004generalized}.

A key development in the study of regularization by fBm was initiated by Nualart and Ouknine~\cite{nualart2002regularization}, who extended the foundational works of Zvonkin \cite{zvonkin1974transformation} on Brownian-driven SDEs and established a phase transition: when \( H \leq 1/2 \), the drift is \enquote{singular}, existence and uniqueness results only require a linear growth condition on the drift function, whereas for \( H > 1/2 \), the drift is \enquote{regular}, requiring continuity of the drift for existence and at least H\"older regularity of the drift for classical well-posedness results. This phase transition was further refined by Catellier and Gubinelli \cite{catellier}, who extended these results to arbitrary dimensions. Since then, a substantial body of literature has emerged, further analyzing the interplay between fBm and singular drift. The stochastic sewing technique was introduced by Khoa L\^e \cite{le2020stochastic}, arguably one of the most influential papers in this domain. 

For \( H > 1/2 \) almost all results assume that the drift function is at least H\"older continuous, in stark contrast to the case \( H < 1/2 \), where significant progress has been made for more singular drifts, including those with mixed integrability conditions or interpreted as distributions.

The study of SDEs with discontinuous drift is highly challenging, with only a few known results if \( H > 1/2 \). Mishura and Nualart \cite{mishuranualart} considered the specific case where the drift is the right-continuous sign function (with the convention $\operatorname{sign}(0) = 1$), proving weak existence under the constraint \( H < (1+\sqrt{5})/4 \), which was later removed by Boufoussi and Ouknine \cite{boufoussi}.  However, all of these works assume that the drift is at least left- or right-continuous. More recently, Butkovsky et al.\ \cite{butkovsky0,butkovsky2} provided weak existence results in a broader setting involving drifts satisfying mixed integrability conditions. To the best of our knowledge, no results so far establish strong existence and uniqueness for SDEs driven by fBm with \( H > 1/2 \) and discontinuous drift (that is neither left- nor right-continuous).

In this paper, we aim to make progress in this direction by adopting a probabilistic approach to analyze SDEs with singular drifts in this regime. We consider more generally a broader class of Volterra processes, which admit an integral representation with respect to standard Brownian motion. Under certain integrability conditions on the variance function of such processes, we derive novel existence and uniqueness results. Prime examples include long-range dependent fBm and the well-known Riemann-Liouville process. We give well-posedness results for a collection of singular drift functions. In particular, our results extend the known regularization effects of fBm with \( H > 1/2 \) beyond the H\"older continuous setting.

\subsection{Framework and main results}

In this section, we introduce the precise framework in which we work and present our main results. We consider stochastic differential equations driven by a Volterra process. More precisely, throughout, let \( T \in (0,\infty) \), denote by $(\Omega, (\mathcal{A}_t)_{t\in [0,T]},\mathbb{P})$ a filtered probability space and let \( (W_t)_{t \in [0,T]} \) be a standard one-dimensional Brownian motion, where we assume that $(\mathcal{A}_t)_{t\in [0,T]}$ is the augmented filtration generated by \( (W_t)_{t \in [0,T]} \) satisfying the usual conditions. Let \( K \) be an \( L^2 \) kernel, i.e., a function \( K \colon [0,T] \times [0,T] \to \mathbb{R} \), \((t,s) \mapsto K(t,s)\), such that
\[
    \int_{0}^{T} \int_{0}^{T} K(t,s)^2 \mathrm{d}t \mathrm{d}s < \infty.
\]
For more background on Volterra processes, we refer to \cite{baudoinnualart}. Now, define the associated Volterra process \( (B^K_t)_{t \in [0,T]} \) given by
\[
    B^K_t = \int_0^t K(t,s) \mathrm{d} W_s, \quad t \in [0,T].
\]
Consider the SDE
\begin{align}\label{volterrasde}
\begin{split}
    \mathrm{d}X_t & = b(t,X_t) \mathrm{d}t + \mathrm{d} B^K_t, \quad t \in [0,T], \\
    X_0 & = x_0 ,
\end{split}
\end{align}
where \( b\colon [0,T] \times \mathbb{R} \to \mathbb{R} \) is a measurable function and $x_0 \in \mathbb{R}$ represents some initial condition. We will only need the following assumption regarding an estimate for the kernel function \( K \). This assumption will be imposed throughout this paper. For $t\in (0,T]$ define
\[
    \kappa^2_{\varepsilon} := \int_{t-\varepsilon}^{t} K(t,s)^2 \mathrm{d} s, \quad \varepsilon \in (0,t).
\]
We assume that there exists a parameter \( H \in (0,1) \) and a constant \( c \in (0,\infty) \), independent of \( \varepsilon \) and $t$, such that
\begin{equation} \label{eq:kernel_assumption}
    \kappa_{\varepsilon} \geq c \varepsilon^H.
\end{equation}
As we shall see below, assumption \eqref{eq:kernel_assumption} is fulfilled in particular for the fractional Brownian motion with \( H > 1/2 \), but also for another prominent example.

Now we introduce the first main result from a collection of examples with singular drift, where we obtain strong existence results and well-posedness results, in particular a regularization effect on differential equations that are otherwise ill-posed without the noise. Our first example is motivated by the scenario where the drift \( b \) is the sign function with the convention \( \operatorname{sign}(0) = 0 \), and there are no further assumptions such as right- or left-continuity. We first impose the following assumption:

\begin{assumption} \label{ass:A1}
The function \( b(t,x) \) is piecewise uniformly continuous in the variable \( x \), with finitely many discontinuity points \( a_1, \dots, a_m \in \mathbb{R} \), where \( m \in \mathbb{N} \). Moreover, the function \( b(t,x) \) satisfies the linear growth condition
$$ |b(t, x)| \leq C (1 + |x|)$$
for every $x\in \mathbb{R}$, uniformly in $t\in [0,T]$.
\end{assumption}

\begin{theorem} \label{thm:strong_solution}
Consider the SDE \eqref{volterrasde}. Under Assumption \ref{ass:A1}, the SDE admits a strong solution.
\end{theorem}

We also emphasize that the result is new in the presence of discontinuities. In contrast to \cite{boufoussi} we do not assume the drift to be monotone. Due to an argument presented in \cite{boufoussi}, under an additional assumption on the drift, we obtain the following well-posedness result.

\begin{corollary} \label{cor:unique_solution}
Under Assumption \ref{ass:A1} and the additional assumption that the drift $b(t,x)$ is nonincreasing in \( x \) for every $t$, the SDE \eqref{volterrasde} admits a unique strong solution.
\end{corollary}

We now turn to our next example of highly irregular drift functions. This class of functions is motivated by the example where the drift is the Dirichlet function, that is \( b(t,x) = b(x) = \mathbbm{1}_{\mathbb{Q}}(x) \). As is well known, the corresponding ordinary differential equation
\[
    {x}_t = \int_0^t b(x_s)  \mathrm{d}s + x_0,
\]
admits the unique solution \( x_t = x_0 \) if the initial condition \( x_0 \in \mathbb{R} \setminus \mathbb{Q} \), but is ill-posed if \( x_0 \in \mathbb{Q} \), since in this case no solution exists. However, as we show below, for a generalization of such functions, the behavior of the perturbed equation is markedly different.

To describe this generalization, we introduce the following assumption:

\begin{assumption} \label{ass:A2}
There exist \( m \in \mathbb{N} \), countable sets \( M_1, \dots, M_m \subset \mathbb{R} \), and functions \( f_1, \dots, f_m \) with \( f_i \in L^{q_i}(\mathbb{R}) \), for \( q_1, \dots, q_m \in (1, \infty] \), such that
\[
    b(t,x) = b(x) = \sum_{i=1}^m f_i(x) \mathbbm{1}_{M_i}(x), \quad x \in \mathbb{R}.
\]
\end{assumption}

Now we state the next main theorem:

\begin{theorem} \label{thm2:strong_solution}
Consider the SDE \eqref{volterrasde}. Under Assumption \ref{ass:A2}, the SDE admits a strong solution. Furthermore, if the functions \( f_i \) in Assumption \ref{ass:A2} are bounded, the solution is unique.
\end{theorem}

We emphasize that, as will be shown in the proof below, the existence of a solution is relatively straightforward. The more subtle aspect, and what makes this result nontrivial, is the uniqueness of the solution. The proof of uniqueness follows from the methods developed in this paper and will be detailed in the next sections.

We now deliver a final example of singular drift functions that can be regularized. This example is motivated by taking \( b(x) = \mathbbm{1}_{\mathbb{R} \setminus \mathbb{Q}}(x) \). Again, the corresponding ordinary differential equation is ill-posed: for example, if the initial condition \( x_0 \in \mathbb{Q} \), then there is more than one solution to this equation.

Now, we state the next main theorem. The proof of this result follows similar ideas and observations as in Theorem \ref{thm2:strong_solution}, making use of the methods and techniques developed here.

\begin{theorem} \label{thm3:strong_solution}
Consider the SDE \eqref{volterrasde}. Assume that \( b(x) = \mathbbm{1}_{M}(x) \), \( x \in \mathbb{R} \), for some set $M\subset \mathbb{R}$ such that $\mathbb{R} \setminus M$ is countable. Then the SDE admits a unique strong solution. More generally, if \( b(t,x) = \tilde{b}(t,x) \mathbbm{1}_{M}(x) \), where \( \tilde{b} \) is a measurable, uniformly bounded function such that the SDE with \( b \) replaced by \( \tilde{b} \) admits a unique strong solution, then the SDE with \( b \) also admits a unique strong solution.
\end{theorem}

\subsection{Leading Examples}\label{sec1.1:leading}

In this section, we provide the promised examples that fall within the class of Volterra processes considered in this paper. We begin with the fBm.

\subsubsection{Long-range dependent fractional Brownian motion}

Let \( (B_t^H )_{ t \in [0,T] } \) be a fractional Brownian motion with Hurst parameter \( H \in (0,1) \), that is a centered Gaussian process with covariance function
\[
    \mathbb{E}(B^H_t B^H_s) = \frac{1}{2} \left( |t|^{2H} + |s|^{2H} - |t - s|^{2H} \right), \quad s,t \in [0,T].
\]
It is well known that there exists a standard Brownian motion \( (W_t)_{t \in [0,T]} \) such that the fBm \( (B_t^H )_{ t \in [0,T] } \) has an integral representation of the form
\[
    B^H_t = \int_0^t K_H(t, s) \, \mathrm{d} W_s,
\]
where the Volterra kernel is given by
\[
    K_H(t, s) = \Gamma(H + 1/2)^{-1} (t - s)^{H - 1/2} F(H - 1/2, 1/2 - H, H + 1/2, 1 - t/s)
\]
and $F$ denotes the Gauss hypergeometric function. We recall from \cite{decreusefond1999stochastic} the following representation of the Volterra kernel of the fBm for $H \in (1/2,1)$:
\[
    K_H(t,s) = \frac{s^{1/2-H}}{\Gamma(H - 1/2)} \int_s^t u^{H-1/2} (u-s)^{H-3/2} \mathrm{d}u \mathbbm{1}_{[0,t]}(s) \,.
\]
Since $u^{H-1/2} \ge s^{H-1/2}$ for all $u \in [s,t]$, the kernel is bounded from below by
\begin{equation} \label{fbmkernelbound}
     |K_H(t, s)| \geq C(H) (t - s)^{H - 1/2} \mathbbm{1}_{[0,t]}(s),
\end{equation}
where \( C(H) = \Gamma(H - 1/2)^{-1} (H - 1/2)^{-1} \) represents a positive and finite constant.

%%% alter text:
%bounds on the kernel \( K_H(t,s) \) for every \( t,s \in [0,T] \). Namely, for \( H \in (0, 1) \), we have
%\[
%    |K_H(t, s)| \leq C_1(H) s^{-|H - 1/2|} (t - s)^{\max\{-1/2 + H, 0\}} \mathbbm{1}_{[0,t]}(s),
%\]
%and for \( H \in (1/2, 1) \), we have
%\begin{equation} \label{fbmkernelbound}
 %     |K_H(t, s)| \geq C_2(H) (t - s)^{H - 1/2} \mathbbm{1}_{[0,t]}(s),
%\end{equation}
%where \( C_1(H) \) and \( C_2(H) \) represent positive and finite constants defined in Lemma 3.1 and Theorem 3.2 of \cite{decreusefond1999stochastic}.

From \eqref{fbmkernelbound}, it follows that the assumption \eqref{eq:kernel_assumption} is fulfilled for \( H > 1/2 \). Thus, we obtain the following statement:

\begin{theorem} \label{thm:fbm_sde_solution}
Let \( (B_t^H)_{t \in [0,T]} \) be a fractional Brownian motion with Hurst index \( H \in (1/2, 1) \). Consider the SDE
\begin{align}\label{fbmsde}
\begin{split}
    \mathrm{d}X_t &= b(t,X_t) \, \mathrm{d} t + \mathrm{d} B_t^H, \quad t\in [0,T],\\
     X_0 &= x_0 \in \mathbb{R} ,
\end{split}
\end{align}
where $b \colon [0,T] \times \mathbb{R} \to \mathbb{R}$ is a measurable function. Then the following assertions hold:

(i) If \( b \) satisfies Assumption \ref{ass:A1}, there exists a strong solution to \eqref{fbmsde}. Moreover, if \( b(t,x) \) is non-increasing in \( x \) for every \( t \), the solution is unique.

(ii) If \( b \) satisfies Assumption \ref{ass:A2}, there exists a strong solution to \eqref{fbmsde}. Moreover, if the functions \( f_i \) in Assumption \ref{ass:A2} are bounded, the solution is unique.

(iii) Assume that \( b(x) = \mathbbm{1}_{M}(x) \), \( x \in \mathbb{R} \), for some set $M\subset \mathbb{R}$ such that $\mathbb{R} \setminus M$ is countable. Then the SDE \eqref{fbmsde} admits a unique strong solution. More generally, if \( b(t,x) = \tilde{b}(t,x) \mathbbm{1}_{M}(x) \), where \( \tilde{b} \) is a measurable, uniformly bounded function such that the SDE \eqref{fbmsde} with \( b \) replaced by \( \tilde{b} \) admits a unique strong solution, then the SDE with \( b \) also admits a unique strong solution.
\end{theorem}

It is worth mentioning that, as an application of Theorem \ref{thm:fbm_sde_solution}, we obtain well-posedness results for stochastic differential equations with multiplicative noise of the form
\begin{align} \label{eq:multiplicative_sde}
\begin{split}
    \mathrm{d}X_t &= b(t, X_t)\, \mathrm{d} t + \sigma(X_t)\, \mathrm{d} B_t^H, \quad t\in [0,T] ,\\
     X_0 &= x_0 \in \mathbb{R}.
\end{split}
\end{align}
This approach is specific to the case \(H > 1/2\).

To ensure well-posedness, we impose that $b$ is bounded and \(\sigma\) is a Hölder continuous function of order \(\delta > \frac{1}{H} - 1\), that is there exists a constant \(C > 0\) such that for all \(x, y \in \mathbb{R}\),
\[
    |\sigma(x) - \sigma(y)| \leq C |x - y|^{\delta}.
\]
Under this assumption, the stochastic integral
\[
    \int_0^t \sigma(X_s)\, \mathrm{d} B_s^H
\]
exists pathwise. Moreover, a change of variable formula for such integrals, as developed in \cite{zahle}, provides a useful tool for the analysis of these equations. Additionally, assume that \(\sigma\) is bounded away from zero, i.e., there exists a constant \(c > 0\) such that
\[
    |\sigma(z)| \geq c \quad \text{for all } z \in \mathbb{R}.
\]
Consider the transformation
\[
    F(x) = \int_0^x \frac{1}{\sigma(z)}\, \mathrm{d} z.
\]
Applying the change of variable formula \cite[Theorem 4.3.1]{zahle} to \(F(X_t)\), where \(X_t\) is as in \eqref{eq:multiplicative_sde}, allows us to transform the equation into one with additive noise, thereby facilitating the analysis of well-posedness. More precisely,  \(( X_t)_{t\in [0,T]}\) is a solution to \eqref{eq:multiplicative_sde} if and only if the process \((Y_t = F(X_t))_{t\in [0,T]}\) is a solution to
\begin{align*}
    \mathrm{d}Y_t &= \frac{b(t, F^{-1}(Y_t))}{\sigma(F^{-1}(Y_t))}  \, \mathrm{d} t + \mathrm{d} B_t^H, \quad t \in [0,T], \\
    Y_0 &= F(x_0).
\end{align*}
The transformed drift function $\frac{b(t, F^{-1}(Y_t))}{\sigma(F^{-1}(Y_t))}$ is bounded and piecewise uniformly continuous because $1/\sigma$, $F^{-1}$ are uniformly continuous and $b$ is bounded and piecewise uniformly continuous. As a consequence, we obtain the following result, which extends the findings of \cite{catellier2,dareiotis} to further examples of singular drift functions.

\begin{corollary} \label{cor:multiplicative_sde}
Consider the SDE \eqref{eq:multiplicative_sde} and assume that $b$ is bounded and piecewise uniformly continuous, and $\sigma$ is H\"older continuous of order $\delta > 1/H-1$ and that there are constants $0<c<C<\infty$ such that $c \leq |\sigma(x)| \le C$ for all $x \in \mathbb{R}$. Then the first conclusion of Theorem \ref{thm:fbm_sde_solution} (i), as well as the conclusions of (ii) and (iii), hold. Moreover, if the function \( \sigma \) is increasing, the second conclusion of (i) also holds.
\end{corollary}

We remark that we invoke a change of variable formula given in \cite[Theorem 4.3.1]{zahle} for the corollary above which is valid for arbitrary processes with H\"older continuous paths of order $\gamma > 1/2$. Thus the statement holds for such a larger class of Volterra processes.

\subsubsection{Riemann--Liouville process}

We now treat the Riemann--Liouville process \( (R_t^H)_{t \in [0,T]} \), which is defined via the stochastic integral representation  
\[
    R_t^H = \int_0^t (t-s)^{H-1/2} \, \mathrm{d} W_s, \quad t \in [0,T],
\]
where \( H \in (0,1) \) is a given parameter. This process arises in the Mandelbrot–Van Ness representation of fractional Brownian motion and has been studied in various contexts, see \cite{chen, lifshits} for example.  

The variance function of the Riemann--Liouville process is given by  
\[
    \kappa_t^2 = \int_0^t (t-s)^{2H-1} \, \mathrm{d} s = \frac{t^{2H}}{2H}.
\]
Moreover, assumption \eqref{eq:kernel_assumption} is clearly satisfied for every \( H \in (0,1) \). Thus, if we consider the SDE  
\begin{align} \label{eq:rl_sde}
\begin{split}
    \mathrm{d}X_t &= b(t,X_t) \, \mathrm{d} t + \mathrm{d} R_t^H, \quad t\in [0,T],\\
     X_0 &= x_0 \in \mathbb{R},
\end{split}
\end{align} 
we obtain the following result, which extends some of the findings in \cite{harang} to arbitrary \( H \in (0,1) \), without the restriction \( H < 1/2 \).  

\begin{theorem} \label{thm:rl_sde_solution}
Let \( (R_t^H)_{t \in [0,T]} \) be the Riemann--Liouville process with parameter $H\in (0,1)$. The conclusions of Theorem \ref{thm:fbm_sde_solution} also hold for the SDE \eqref{eq:rl_sde}.
\end{theorem}

Further examples where our results apply include mixtures of so-called completely correlated mixed fractional Brownian motions, as introduced in \cite{aurzada}. Specifically, consider the process  
\[
    B_t^K = \sum_{i=1}^{n} B_t^{H_i},
\]
where \( (B_t^{H_i})_{t \in [0,T]} \), \( i=1, \ldots, n \), $n \in \mathbb{N}$, are completely correlated fractional Brownian motions with Hurst parameters \( H_1, \ldots, H_n \in (0,1) \). Then, the conclusions of our results also hold for this process if \( \max_{1 \leq i \leq n} H_i \geq 1/2 \), thus extending results in \cite{nualart2022regularization}. Moreover, the conclusions also hold for a (well-defined) linear combination of infinitely many completely correlated fractional Brownian motions.

The rest of this work is organized as follows. In Section \ref{sec:2strong}, we prove Theorem \ref{thm:strong_solution} and Corollary \ref{cor:unique_solution}. In Section \ref{sec:3dirich}, we prove Theorems \ref{thm2:strong_solution} and \ref{thm3:strong_solution}. Finally, in Section \ref{sec:4weak}, we focus on stochastic differential equations driven by fractional Brownian motion with Hurst index \( H > 1/2 \). Here, independently of the results from the previous sections, we prove that a weak solution of equations with bounded, piecewise Lipschitz drift arises as a limit in a certain Besov space of the solution to a mixed SDE driven by two fBms with different Hurst indices \( H_1 \) and \( H_2 \), where \( H_1 \leq 1/2 \) and \( H_2 = H \). Most of the techniques in this work extend naturally to multidimensional SDEs, with appropriate modifications to account for the vector-valued setting.

\section{Strong solutions of equations with piecewise continuous drift}\label{sec:2strong}

Let us briefly explain how we prove Theorem~\ref{thm:strong_solution}. As indicated earlier, we illustrate a probabilistic aspect of regularization by noise. Roughly speaking, due to the noise arising from the Volterra process, which is Gaussian, the solution to the equation spends little time around the discontinuity points of the drift. More precisely, a solution is expected to have Gaussian-like densities, though we do not aim to prove such a result here (a parallel project \cite{buthenhoff} focuses more concretely on such results). Based on this intuition, in order to prove Theorem \ref{thm:strong_solution}, we establish an inequality that is nearly of Krylov-type. This, together with a comparison theorem, allows us to construct strong solutions, similar to \cite[Section 4]{nualart2002regularization}. A crucial difference to \cite{nualart2002regularization} is that we do not rely on Girsanov’s theorem. Instead, we construct an auxiliary process that is conditionally Gaussian and approximates the solution to the SDE \eqref{volterrasde}.

We proceed with the following definition.

\begin{definition}[Conditionally Gaussian process]
    Let $(X_t)_{t \in [0,T]}$ be a process satisfying \eqref{volterrasde}, $t\in (0,T]$ and $\varepsilon \in (0,t)$. We define the conditionally Gaussian process (CGP) as
    \begin{align}
        Y_t(\varepsilon) \coloneqq X_{t-\varepsilon} +  ({B_t^{K} - B_{t-\varepsilon}^{K}}) .
        \label{eq:definition-Y}
    \end{align}
\end{definition}
\begin{remark}\label{remark:expression-CGP}
    Note that the CGP admits the following expression
    \begin{align*}
        Y_t(\varepsilon) = X_t - \int_{t-\varepsilon}^t b(s,X_s) \mathrm{d}s \,,
    \end{align*}
    i.e. it approximates $X_t$ for $\varepsilon \to 0$.
\end{remark}
\begin{lemma}[cf.\ \citewithin{Lemma 4.4}{fan2021moment}, \citewithin{Lemma 4.5}{sonmez2020mixed}]\label{lemma:X-Y-expectation}
    If $b$ is uniformly bounded, i.e.
    $$ \norm{b}_\infty = \sup_{t\in [0,T]} \sup_{x \in \mathbb{R}} |b(t,x)| < \infty,$$
    then we have 
    \begin{align*}
        \mathbb{E}\kle{|X_t - Y_t(\varepsilon)|} \le \norm{b}_\infty \varepsilon \,
    \end{align*}
    for every $t\in (0,T]$.
\end{lemma}
\begin{proof}
    This is a direct consequence of Remark \ref{remark:expression-CGP}.
\end{proof}
The following lemma justifies the notion \enquote{conditionally Gaussian process}.
\begin{lemma}[cf.\ \citewithin{Lemma 4.2}{fan2021moment}, \citewithin{Lemma 4.4}{sonmez2020mixed}]\label{lemma:Y-conditionally-gaussian}
    Fix $t\in (0,T]$, $\varepsilon \in (0,t)$ and define 
    \begin{equation}
        \xi_t = X_{t-\varepsilon} + \int_0^{t-\varepsilon}   \klr{K(t,s) - K(t-\varepsilon,s)}\mathrm{d}W_s \,.\label{eq:def-xi}
  \end{equation}
    Then, for all $u \in \mathbb{R}$ we have
    \begin{align*}
        \mathbb{E}[\exp\klr{iu Y_t(\varepsilon)} | \mathcal{A}_{t-\varepsilon}] = \exp \klr{iu\xi - \frac{u^2}{2} \int_{t-\varepsilon}^t K(t,s)^2 \mathrm{d}s\,},
    \end{align*}
    i.e.\ given $\mathcal{A}_{t-\varepsilon}$ the CGP $Y_t(\varepsilon)$ is conditionally Gaussian with mean $\xi_t$ and variance 
    \begin{align}
        \kappa_\varepsilon^2 = \int_{t-\varepsilon}^t K(t,s)^2 \mathrm{d}s \,.
        \label{eq:variance-conditionally-gaussian-y}
    \end{align}
\end{lemma}
\begin{proof}
    Since $\int_{t-\varepsilon}^t K(t,s) \mathrm{d}W_s$ is independent of $\mathcal{A}_{t-\varepsilon}$ we have
    \begin{align*}
        &\mathbb{E}\kle{\exp\klr{iu \int_{t-\varepsilon}^t { K(t,s)} \mathrm{d}W_s} | \mathcal{A}_{t-\varepsilon}}  = \exp\klr{-\frac{u^2}{2} \int_{t-\varepsilon}^t { K(t,s)}^2 \mathrm{d}s} \,.
    \end{align*}
    Now, by definition of $(B_t^K)_{t\in [0,T]}$ we have
    \begin{align*}
        B_t^K - B_{t-\varepsilon}^K = \int_0^{t-\varepsilon} \klr{K(t,s) - K(t-\varepsilon,s)} \mathrm{d}W_s + \int_{t-\varepsilon}^t K(t,s) \mathrm{d}W_s \,.
    \end{align*}
    Therefore we can express $Y_t(\varepsilon)$ as
    \begin{align*}
        Y_t(\varepsilon) = \xi_t +\int_{t-\varepsilon}^t K(t,s) \mathrm{d}W_s \,,
    \end{align*}
    where $\xi_t$ is defined in equation \eqref{eq:def-xi} and measurable with respect to $\mathcal{A}_{t-\varepsilon}$. Therefore we end up with the desired result:
    \begin{align*}
        \mathbb{E}[\exp\klr{iuY_t(\varepsilon)} | \mathcal{A}_{t-\varepsilon}] &= \exp\klr{iu\xi} \mathbb{E}\kle{\exp\klr{iu \int_{t-\varepsilon}^t {  K(t,s)} \mathrm{d}W_s} | \mathcal{A}_{t-\varepsilon}} \\
        &= \exp\klr{iu\xi - \frac{u^2}{2} \int_{t-\varepsilon}^t { K(t,s)}^2 \mathrm{d}s} \,.
    \end{align*}
\end{proof}
%
%By extending the previous proof, one can show the following remark:
%
%\begin{remark}\label{remark:YY-cond-gaussian}
%    Let $H_1 \in (0,1/2)$, $H_2 \in (1/2,1)$ and fix $a_1,a_2,a_1',a_2' \in \mathbb{R}\setminus\klg{0}$. Let $X'$ satisfy the SDE \eqref{eq:sde} with diffusion coefficients $a_1'$ and $a_2'$, i.e.,
    %
%    \begin{align*}
 %       X_t' = x_0 + \int_0^t b(s,X_s') \mathrm{d}s + a_1' B_t^{H_1} + a_2' B_t^{H_2}\,, \quad t \in [0,T]\,.
  %  \end{align*}
    %
   % Define $Y'(\varepsilon)$, $\xi'$ and $(\kappa')^2$ analogously to equations \eqref{eq:definition-Y}, \eqref{eq:def-xi} and \eqref{eq:variance-conditionally-gaussian-y}. Furthermore, let
    %
    %\begin{align*}
     %   \lambda \coloneqq \int_{T-\varepsilon}^T \klr{a_1 a_1' \klr{K_{H_1}(T,s)}^2 + \klr{a_1'a_2 + a_1a_2'} K_{H_1}(T,s) K_{H_2}(T,s) + a_2 a_2' \klr{K_{H_2}(T,s)}^2} \mathrm{d}s \,.
    %\end{align*}
    %
    %Then, given $\mathcal{A}_{T-\varepsilon}$, the random vector $(Y_t(\varepsilon),Y_t'(\varepsilon))_{t \in [0,T]}$ is conditionally Gaussian with mean $(\xi,\xi')$ and covariance matrix
    %
    %\begin{align}
    %    \Sigma = \matr{\kappa^2 & \lambda\\ \lambda & (\kappa')^2} \,.
  %  \end{align}
%\end{remark}

As a consequence of Remark \ref{remark:expression-CGP} and Lemma \ref{lemma:Y-conditionally-gaussian}, we have the following. Recall at this point that the parameter $H\in (0,1)$ is derived from \eqref{eq:kernel_assumption}.
\begin{lemma}\label{lemma:bound-probability-small-interval}
    Let $(\alpha_n)_{n \in \mathbb{N}}$ denote any sequence in $(0,1)$, let $(X_t)_{t \in [0,T]}$ be a process satisfying \eqref{volterrasde} and $t\in (0,T]$. If $b$ is uniformly bounded, there exists a constant $C > 0$ independent of $n,t$ such that
    \begin{align*}
        \mathbb{P}(X_t \in (x, x + \alpha_n)) \le C t^{-H} \alpha_n^{1-H} 
    \end{align*}
    for all $x \in \mathbb{R}$.
\end{lemma}
\begin{proof}
    In this proof we let $c, C_1, C_2 \ldots$ denote unspecified positive constants. Since the drift function $b$ is bounded, we have due to Remark \ref{remark:expression-CGP} with $\varepsilon = t \alpha_n$
    \begin{align*}
        |X_t - Y_t(t\alpha_n)| \le ct\alpha_n \,.
    \end{align*}
    Therefore, $X_t \in (x, x+\alpha_n)$ implies for the CGP $Y_{t}(t\alpha_n) \in (x-ct\alpha_n, x + (1+ct)\alpha_n)$ and we have
    \begin{align*}
        \mathbb{P}(X_t \in (x, x + \alpha_n)) &\le \mathbb{P}(Y_t(t\alpha_n) \in (x-ct\alpha_n, x + (1+ct)\alpha_n)) \\
        &= \mathbb{E}\kle{\mathbbm{1}{\klg{Y_t(t\alpha_n) \in (x-ct\alpha_n, x + (1+ct)\alpha_n)}}} \\
        & = \mathbb{E}\kle{\mathbb{E}\kle{\mathbbm{1}{\klg{Y_t(t\alpha_n) \in (x-ct\alpha_n, x + (1+ct)\alpha_n)}}}\mid \mathcal{A}_{t-t\alpha_n} }.
    \end{align*}
    According to Lemma \ref{lemma:Y-conditionally-gaussian} $Y_t(t\alpha_n)$ is conditionally Gaussian with mean $\xi_t$ and variance $\kappa_\varepsilon^2$ given in equations \eqref{eq:def-xi} and \eqref{eq:variance-conditionally-gaussian-y}, respectively. Denote by $f(z)$ the density of the Gaussian distribution $\mathcal{N}(0, \kappa_\varepsilon^2)$, then we have
    \begin{align*}
        &\mathbb{E}\kle{\mathbbm{1}{\klg{Y_t(t\alpha_n) \in (x-ct\alpha_n, x + (1+ct)\alpha_n)}}\mid \mathcal{A}_{t-t\alpha_n} } \\
        &= \int_{x - ct\alpha_n}^{x + (1+ct)\alpha_n} f(z - \xi_t) \mathrm{d}z = \int_{x - \xi_t - ct\alpha_n}^{x - \xi_t + (1+ct)\alpha_n} f(z) \mathrm{d}z \,.
    \end{align*}
    Due to \eqref{eq:kernel_assumption} we can estimate the variance $\kappa_\varepsilon^2$ as
    \begin{align*}
        \kappa_\varepsilon^2 \ge C_1 t^{2H} \alpha_n^{2H} \,,
    \end{align*}
    which implies that $f(z) \le C_2 t^{-H} \alpha_n^{-H}$, because $f$ is a Gaussian density. Hence,
    \begin{align*}
        &\mathbb{E}\kle{\mathbb{E}\kle{\mathbbm{1}{\klg{Y_t(t\alpha_n) \in (x-ct\alpha_n, x + (1+ct)\alpha_n)}}}\mid \mathcal{A}_{t-t\alpha_n} }\\
        & \le C_3 \klr{t + C_4} t^{-H} \alpha_n^{1-H} \le C_5 t^{-H} \alpha_n^{1-H}
    \end{align*}
    and we can conclude
    \begin{align*}
        \mathbb{P}(X_t \in (x, x+\alpha_n)) \le C t^{-H} \alpha_n^{1-H} 
    \end{align*}
    as needed.
\end{proof}
\begin{remark} \label{remark:probzero}
    Because $H \in (0,1)$, we have $\mathbb{P}(X_t \in (x, x+\alpha_n)) \to 0$ for $n \to \infty$ if $(\alpha_n)_{n \in \mathbb{N}}$ represents a zero sequence. In particular, every solution to \eqref{volterrasde} with bounded drift must satisfy $\mathbb{P}(X_t =x ) = 0$ for every $x \in \mathbb{R}$ and $t\in (0,T]$.
\end{remark}

Now we establish the following bound which is nearly of Krylov type. Though we have similar assumptions on the exponents as in \cite[Proposition 6]{nualart2002regularization}, it is a (slightly) different setting.

\begin{proposition}\label{prop:krylov-bound}
    Suppose that \( b \) is uniformly bounded. Let $(X_t)_{t \in [0,T]}$ be a process satisfying \eqref{volterrasde}. In addition, let \(q>  1 + H \) and let \( g \colon [0, T] \times \mathbb{R} \to \mathbb{R} \) be a measurable, non-negative, bounded function, that is piecewise uniformly continuous in $x$ with finitely many discontinuity points. There exist constants \( C_1, C_2 \in (0, \infty)\) with $C_1$ depending only on \( T \), \( \|b\|_\infty \), \( \|g\|_\infty \) and \( q \), $C_2$ depending only on \( T \), \( \|b\|_\infty \) and \( q \) such that we have
    \[
    \mathbb{E} \left[ \int_0^T g(t, X_t) \, \mathrm{d}t \right] \leq  C_1n^{H-1} + {C_2 \varepsilon_n} +  C_2 n^{2}
    \left( \int_0^T \int_{\mathbb{R}} g(t, x)^q \, \mathrm{d}x \, \mathrm{d}t \right)^{1/q},
    \]
    where $\varepsilon_n$ is a zero sequence independent of $t$.
\end{proposition} 
\begin{proof}
    Without loss of generality assume that $g$ has one discontinuity point in $x$ at zero. As before, in this proof we denote by $C$ an unspecified positive and finite constant. We have
    \begin{align*}
    & \mathbb{E} \left[ \int_0^T g(t, X_t) \, \mathrm{d}t \right]  =   \mathbb{E} \left[ \int_0^T g(t, X_t) \mathbbm{1}_{\{0 < |X_t| < 1/n\}} \, \mathrm{d}t \right] \\
     & \quad +  \mathbb{E} \left[ \int_0^T g(t, X_t) \mathbbm{1}_{\{X_t > 1/n\}} \, \mathrm{d}t \right] + \mathbb{E} \left[ \int_0^T g(t, X_t) \mathbbm{1}_{\{X_t < -1/n\}} \, \mathrm{d}t \right] .
     \end{align*}
    Since $g$ is bounded, we can estimate
    \begin{align*}
    \mathbb{E} \left[ \int_0^T g(t, X_t) \mathbbm{1}_{\{0 < |X_t| < 1/n\}} \, \mathrm{d}t \right] & \leq C \mathbb{E} \left[ \int_0^T  \mathbbm{1}_{\{-1/n < X_t < 1/n\}} \, \mathrm{d}t \right] \\
    & =  C \int_0^T  \mathbb{P} ({-1/n < X_t < 1/n} ) \, \mathrm{d}t \\
    & \leq  C \int_0^T  n^{H-1} t^{-H} \, \mathrm{d}t \leq Cn^{H-1},
     \end{align*}
    where we used Lemma \ref{lemma:bound-probability-small-interval} and the fact that $H<1$. Next we estimate 
    \begin{align*}
        \mathbb{E} \left[ \int_0^T g(t, X_t) \mathbbm{1}_{\{X_t > 1/n\}} \, \mathrm{d}t \right] .
    \end{align*}
    The other term is estimated the same way. Recall again the definition and properties of $Y_t (t/n^2)$ and that
    $ |X_t - Y_t (t/n^2)| \leq Cn^{-2}$
    for some constant $C$ independent of $t$. Therefore, $X_t>1/n$ implies $Y_t (t/n^2) > 1/n - C/n^2>0$ for some sufficiently large $n \in \mathbb{N}$. Since $g$ is non-negative and uniformly continuous on $(0,\infty)$, this implies that for sufficiently large $n$ we have $g(t,X_t) \leq g(t,Y_t (t/n^2)) + \varepsilon_n$ where $\varepsilon_n$ is a zero sequence independent of $t$. We obtain, abbreviating $Y_t (t/n^2) = Y_t $,
    \begin{align*}
    \mathbb{E} \left[ \int_0^T g(t, X_t) \mathbbm{1}_{\{X_t > 1/n\}} \, \mathrm{d}t \right] &\leq   \mathbb{E} \left[ \int_0^T g(t, Y_t) \mathbbm{1}_{\{X_t > 1/n\}} \, \mathrm{d}t \right] + T \varepsilon_n \\
     & \leq \mathbb{E} \left[ \int_0^T g(t, Y_t) \, \mathrm{d}t \right] + T \varepsilon_n = \mathbb{E} \left[ \ \mathbb{E} \left[\int_0^T g(t, Y_t) \, \mathrm{d}t \mid \mathcal{A}_{t-t/n^2} \right]\right] + T \varepsilon_n .
     \end{align*}
    Since $Y_t$ is conditionally normal with mean $\xi_t$ and variance $\kappa^2_n = \int_{t-t/n^2}^t K(t,s)^2 \mathrm{d}s$ we have, using Fubini and the nonnegativity of the function $g$,
    \begin{align*}
    \mathbb{E} \left[ \ \mathbb{E} \left[\int_0^T g(t, Y_t) \, \mathrm{d}t \mid \mathcal{A}_{t-t/n^2} \right]\right]  &= \mathbb{E} \left[\int_0^T \ \mathbb{E} \left[ g(t, Y_t) \mid \mathcal{A}_{t-t/n^2} \right]\, \mathrm{d}t \right]  \\ &= \mathbb{E} \left[ \int_0^T \int_{\mathbb{R}} \frac{1}{\sqrt{2\pi}\kappa_n}g(t,y) \exp\Big( -\frac{(y-\xi_t)^2}{2\kappa_n^2}\Big) \mathrm{d}y \mathrm{d}t \right].
     \end{align*}
    Let $p \in (1, 1+ \frac{1}{H})$. Applying Hölder's inequality with \( 1/p + 1/q = 1 \) and \( q >  1 + H \), yields
    \begin{align*}
    &  \mathbb{E} \left[ \int_0^T \int_{\mathbb{R}} \frac{1}{\sqrt{2\pi}\kappa_n}g(t,y) \exp\Big( -\frac{(y-\xi_t)^2}{2\kappa_n^2}\Big) \mathrm{d}y \mathrm{d}t \right] \\
    & \leq \Big( \int_0^T \int_{\mathbb{R}} g(t,y)^q \mathrm{d}y \mathrm{d}t\Big)^{\frac{1}{q}} \cdot C \mathbb{E} \left[ \klr{\int_0^T \int_{\mathbb{R}} \frac{1}{\kappa^p_n} \exp\Big( -\frac{(y-\xi_t)^2}{2\kappa_n^2/p}\Big) \mathrm{d}y \mathrm{d}t }^{\frac{1}{p}}\right] \\
    & = \Big( \int_0^T \int_{\mathbb{R}} g(t,y)^q \mathrm{d}y \mathrm{d}t\Big)^{\frac{1}{q}} \cdot C  \klr{\int_0^T \kappa_n^{1-p} \mathrm{d}t}^{\frac{1}{p}} ,
     \end{align*}
    where by \eqref{eq:kernel_assumption} and the assumption that $p \in (1, 1+ \frac{1}{H})$
    \begin{align*}
     \int_0^T \kappa_n^{1-p} \mathrm{d}t \leq C\int_0^T \Big( t^H/n^{2H} \Big)^{1-p} \mathrm{d}t \leq Cn^{2H(p-1)} \leq Cn^2.
     \end{align*}
    This completes the proof.
\end{proof}

Based on Proposition \ref{prop:krylov-bound} we obtain the following result.

\begin{proposition}\label{prop:helper-existence-strong}
   Assume that there is a sequence \( b_n(t, x) \), $n\in \mathbb{N}$,  of measurable functions, uniformly bounded by \( C \), piecewise uniformly continuous in $x$ with finitely many discontinuity points, such that
    \[
    \lim_{n \to \infty} b_n(t, x) = b(t, x)
    \]
    for almost all \( (t, x) \in [0, T] \times \mathbb{R} \). Suppose also that the corresponding solutions \( X^n_t\) of the equations
    \[
    X^n_t = x + B^K_t + \int_0^t b_n(s, X^n_s) \, \mathrm{d}s, \quad 0 \leq t \leq T,
    \]
    converge almost surely to some process \( X_t \) for all \( t \in [0, T] \). Then the process \( X_t \) is a solution of
    \[
    X_t = x + B^K_t + \int_0^t b (s, X_s) \, \mathrm{d}s, \quad 0 \leq t \leq T.
    \]
\end{proposition} 

\begin{proof}
It suffices to show that
\begin{align*}
    \lim_{n \to \infty} \mathbb{E} \left[ \int_0^T \left| b_n(s, X^n_s) - b(s, X_s) \right| \, \mathrm{d}s \right] = 0 \,.
\end{align*}
We can write
\[
J(n) := \mathbb{E} \left[ \int_0^T \left| b_n(s, X^n_s) - b(s, X_s) \right| \, \mathrm{d}s \right] \leq J_1(n) + J_2(n),
\]
where
\begin{align}
    J_1(n) := \sup_k \mathbb{E} \left[ \int_0^T \left| b_k(s, X^n_s) - b_k(s, X_s) \right| \, \mathrm{d}s \right] \,,
    \label{eq:important-limit}
\end{align}
and
\begin{align*}
    J_2(n) := \mathbb{E} \left[ \int_0^T \left| b_n(s, X_s) - b(s, X_s) \right| \, \mathrm{d}s \right] \,.
\end{align*}
Let \( \varphi \colon \mathbb{R} \to \mathbb{R} \) be a bounded smooth function such that \( 0 \leq \varphi(z) \leq 1 \) for every \( z \), \( \varphi(z) = 0 \) for \( |z| \geq 1 \) and \( \varphi(z) = 1 \) for $|z| < 1/2$. Since $b_n$ is uniformly bounded, the approximating solution is uniformly integrable. Therefore, for every \( \varepsilon > 0 \) we can choose \( R > 0 \) such that
\[
\mathbb{E} \left[ \int_0^T \left| 1 - \varphi \left( \frac{X_t}{R} \right) \right| \, \mathrm{d}t \right] < \varepsilon.
\]
Moreover, the sequence of functions \( b_k \) is relatively compact in \( L^2([0, T] \times [-R, R]) \). Hence, we can find finitely many bounded smooth functions \( F_1, \dots, F_N \) such that for every \( k \),
\[
\int_0^T \int_{-R}^R \left| b_k(t, x) - F_i(t, x) \right|^2 \, \mathrm{d}x \, \mathrm{d}t < \varepsilon
\]
for some \( F_i \). We write

\begin{align*}
&\mathbb{E} \left[ \int_0^T \left| b_k(t, X^n_t) - b_k(t, X_t) \right| \, \mathrm{d}t \right] 
\leq \mathbb{E} \left[ \int_0^T \left| b_k(t, X^n_t) - F_i(t, X_t^n) \right| \, \mathrm{d}t \right] \\
&\quad + \sum_{j=1}^N \mathbb{E} \left[ \int_0^T \left| F_j(t, X_t^n) - F_j(t, X_t) \right| \, \mathrm{d}t \right] \\
&\quad + \mathbb{E} \left[ \int_0^T \left| b_k(t, X_t) - F_i(t, X_t) \right| \, \mathrm{d}t \right] \\
&:= I_1(n, k) + I_2(n) + I_3(k).
\end{align*}

By Proposition \ref{prop:krylov-bound}, we have
\begin{align*}
I_1(n, k) &\leq \mathbb{E} \left[ \int_0^T \varphi \left( \frac{X^n_t}{R} \right) \left| b_k(t, X^n_t) - F_i(t, X^n_t) \right| \, \mathrm{d}t \right] \\
&\quad + \mathbb{E} \left[ \int_0^T \left[ 1 - \varphi \left( \frac{X^n_t}{R} \right) \right] \left| b_k(t, X^n_t) - F_i(t, X^n_t) \right| \, \mathrm{d}t \right] \\
&\leq K l^{H-1} +K \varepsilon_l + K l^{2} \left( \int_0^T \int_{-R}^R \left| b_k(t, x) - F_i(t, x) \right|^2 \, \mathrm{d}x \, \mathrm{d}t \right)^{1/2} \\
&\quad + C \mathbb{E} \left[ \int_0^T \left[ 1 - \varphi \left( \frac{X^n_t}{R} \right) \right] \, \mathrm{d}t \right] \\
& \leq K l^{H-1} +K \varepsilon_l +  Kl^{2} \varepsilon + C\varepsilon
\end{align*}
for some constant \( C \) depending only on \( \|b\|_\infty \) and \( \sup_i \|F_i\|_\infty \). Note that this holds for every $l \in \mathbb{N}$, and every $\varepsilon>0$. Hence, choosing first $l$ sufficiently large, then $\varepsilon$ sufficiently small, we obtain
\[
\lim_{n \to \infty} \sup_k I_1(n, k) \leq \delta,
\]
for arbitrary $\delta>0$. Similarly, we have $I_2(n) \to 0$ using dominated convergence, since $F_j$ are bounded. Analogously, we prove
\[
 \sup_n I_3(n) \leq \delta.
\]
For $I(n,k) \coloneqq I_1(n,k) + I_2(n) + I_3(k)$ it follows that
\[
\lim_{n \to \infty} \sup_k I(n, k) \leq \delta
\]
and this implies that \( \lim_{n \to \infty} J_1(n) = 0 \). For the term \( J_2(n) \), we can write
\begin{align*}
    J_2(n) &= \mathbb{E} \left[ \int_0^T \varphi \left( \frac{X_t}{R} \right) \left| b_n(t, X_t) - b(t, X_t) \right| \, \mathrm{d}t \right] + \mathbb{E} \left[ \int_0^T \left[ 1 - \varphi \left( \frac{X_t}{R} \right) \right] \left| b_n(t, X_t) - b(t, X_t) \right| \, \mathrm{d}t \right] \\
    &\leq K \left( \int_0^T \int_{-R}^R \left| b_n(t, x) - b(t, x) \right|^2 \, \mathrm{d}x \, \mathrm{d}t \right)^{1/2}
    + C \mathbb{E} \left[ \int_0^T \left[ 1 - \varphi \left( \frac{X_t}{R} \right) \right] \, \mathrm{d}t \right].
\end{align*}
Now we can use the same arguments as before.
\end{proof}

We are now in position to construct strong solutions to \eqref{volterrasde} and complete the proof of Theorem~\ref{thm:strong_solution}.

\begin{proof}[Proof of Theorem \ref{thm:strong_solution}]
    Given that we established Proposition \ref{prop:helper-existence-strong} we can from now on proceed as in \cite[Theorem 8]{nualart2002regularization}. That is, starting from SDEs with Lipschitz continuous drift we can follow the steps given in the proof of the latter in order to approximate and construct solutions to equations with piecewise continuous drift satisfying linear growth by applying Proposition \ref{prop:helper-existence-strong} iteratively. More precisely, we use ordered mollified approximations of the drift, which allow one to apply the usual comparison principle for ordinary differential equations.
\end{proof}

It remains to prove Corollary \ref{cor:unique_solution}. As mentioned earlier, this is due to an argument based on \cite[Theorem 3.4]{boufoussi}.

\begin{proof}[Proof of Corollary \ref{cor:unique_solution}]
    Assume that both $(X^1_t)_{t \in [0,T]}$ and $(X^2_t)_{t \in [0,T]}$ are solutions to \eqref{volterrasde}. Then $(X^1_t-X^2_t)_{t \in [0,T]}$ is a continuous semi-martingale. Applying Tanaka's formula, see \cite[Theorem 3.7.1]{karatzas}, we then get
    $$|X_t^1 - X_t^2| = \int_0^t \operatorname{sign} (X_s^1 - X_s^2) \klr{b(s,X_s^1) - b(s,X_s^2)} \mathrm{d} s\leq 0$$
    for every $t\in [0,T]$, where the last inequality follows from the assumption that $b(s,x)$ is nonincreasing in $x$. Thus, we can deduce that $X_t^1 = X_t^2$ for every $t\in [0,T]$, with probability one.
\end{proof}

\section{Strong and unique solutions of equations with Dirichlet-type drift}\label{sec:3dirich}

The goal of this section is to prove Theorem~\ref{thm2:strong_solution} and Theorem~\ref{thm3:strong_solution}. Our approach relies on the results established in Section \ref{sec:2strong}. We begin with the proof of Theorem~\ref{thm2:strong_solution}. As indicated, the construction of a solution is straightforward. In fact, under Assumption~\ref{ass:A2}, a solution to the SDE~\eqref{volterrasde} is given explicitly by $(B_t^K+x_0)_{t\in[0,T]}$ itself. The more subtle and interesting part is to prove that this is the unique solution. To achieve this, we use the results from the previous section to show that any solution to the SDE~\eqref{volterrasde} must necessarily be equal to $B_t^K+x_0$ for every $t\in[0,T]$. The proof of Theorem~\ref{thm3:strong_solution} follows similar ideas. 

We now proceed with the proof of Theorem~\ref{thm2:strong_solution}.

\begin{proof}[Proof of Theorem~\ref{thm2:strong_solution}]
As mentioned, we first show that $X_t = x_0 + B_t^K$, $t\in [0,T]$, solves equation~\eqref{volterrasde}. For this, it suffices to show that 
\[
\int_0^t b(s,B_s^K +x_0) \mathrm{d} s = 0 \quad \text{a.s.}
\]
We estimate
\begin{align*}
\mathbb{E} \left[ \left| \int_0^t b(s,B_s^K+x_0) \mathrm{d}s \right| \right]
&\leq \int_0^t \sum_{i=1}^m \mathbb{E} \left[ |f_i(B_s^K+x_0)| \mathbbm{1}_{M_i}(B_s^K+x_0) \right] \mathrm{d}s.
\end{align*}
Since $f_i \in L^{q_i}(\mathbb{R})$ with $q_i > 1$, $i=1,\dots,m$, and $H \in (0,1)$, for each $q_i$ we can choose $q_i(1), q_i(2) \in (1,\infty)$ with $q_i = q_i(1) \cdot q_i(2)$. Moreover, it holds that 
$$q_i(1) \cdot q_i(2) \cdot H \in (H, q_i).$$ 
Further let $q_i(3) \in (1,q_i(1))$ and denote by $p_i(1), p_i(2), p_i(3)$ the corresponding H\"older conjugates. Denote by $c$ an unspecified constant. Applying H\"older's inequality iteratively, and using the fact that $x_0 + B_s^K$ is normally distributed with mean $x_0$ and variance $\kappa^2(s) := \int_0^s K(s,u)^2 \mathrm{d} u$, we get
\begin{align*}
&\int_0^t \sum_{i=1}^m \mathbb{E} \left[ |f_i(B_s^K+x_0)| \mathbbm{1}_{M_i}(B_s^K+x_0) \right] \mathrm{d}s \\
& \leq \int_0^t \sum_{i=1}^m \mathbb{E} \left[ |f_i(B_s^K+x_0)|^{q_i(1)} \right]^{\frac{1}{q_i(1)}} \mathbb{E} \left[ \mathbbm{1}_{M_i} (x_0+B_s^K) \right]^{\frac{1}{p_i(1)}} \mathrm{d}s \\
& \leq \sum_{i=1}^m \left( \int_0^t \mathbb{E} \left[ |f_i(B_s^K+x_0)|^{q_i(1)} \right]^{\frac{q_i(3)}{q_i(1)}} \mathrm{d}s \right)^{\frac{1}{q_i(3)}} 
\times \left( \int_0^t \mathbb{E} \left[ \mathbbm{1}_{M_i} (x_0+B_s^K) \right]^{\frac{p_i(3)}{p_i(1)}} \mathrm{d}s \right)^{\frac{1}{p_i(3)}},
\end{align*}
with
\begin{align*}
&   \mathbb{E} \left[ |f_i(B_s^K+x_0)|^{q_i(1)} \right]  =\int_\mathbb{R} |f_i(y)|^{q_i(1)} \frac{1}{\sqrt{2\pi}} \kappa(s)^{-1} \exp \left( -\frac{(y-x_0)^2}{2\kappa(s)^2} \right) \mathrm{d}y \\
&\leq \left( \int_\mathbb{R} |f_i(y)|^{q_i(1) q_i(2)} \mathrm{d}y \right)^{\frac{1}{q_i(2)}} 
\times \left( c \int_\mathbb{R} \kappa(s)^{-p_i(2)} \exp \left(-\frac{(y-x_0)^2}{2\kappa(s)^2/p_i(2)} \right) \mathrm{d}y \right)^{\frac{1}{p_i(2)}} \\
& \leq c \kappa (s)^{\frac{1-p_i(2)}{p_i(2)}},
\end{align*}
so that by \eqref{eq:kernel_assumption}
\begin{align*}
& \left( \int_0^t \mathbb{E} \left[ |f_i(B_s^K+x_0)|^{q_i(1)} \right]^{\frac{q_i(3)}{q_i(1)}} \mathrm{d}s \right)^{\frac{1}{q_i(3)}}   \leq C  \klr{ \int_0^t \kappa (s)^{\frac{q_i(3)}{q_i(1)}\frac{1-p_i(2)}{p_i(2)}} \mathrm{d}s}^{\frac{1}{q_i(3)}} < \infty
\end{align*}
for each $i$. Above, we used that $p_i(2)$ is conjugate to $q_i(2)$, the exponent of $\kappa(s)$ equals 
\begin{align*}
    -\frac{q_i(3)}{q_i(1)q_i(2)}=-\frac{1}{q_i(2)}\frac{q_i(3)}{q_i(1)} \,.
\end{align*}
From this we obtain
\begin{align*}
&\int_0^t \sum_{i=1}^m \mathbb{E} \left[ |f(B_s^K+x_0)| \mathbbm{1}_{M_i}(B_s^K+x_0) \right] \mathrm{d}s  \leq  C \sum_{i=1}^m\left( \int_0^t \mathbb{E} \left[ \mathbbm{1}_{M_i} (x_0+B_s^K) \right]^{\frac{p_i(3)}{p_i(1)}} \mathrm{d}s \right)^{\frac{1}{p_i(3)}} \\
& = C \sum_{i=1}^m\left( \int_0^t \mathbb{P}(x_0+B_s^K \in M_i)^{\frac{p_i(3)}{p_i(1)}} \mathrm{d}s \right)^{\frac{1}{p_i(3)}} = 0,
\end{align*}
since 
$x_0+B_s^K$ is Gaussian and $M_i$ is countable by assumption. This proves that
\[
\int_0^t b(s,B_s^K +x_0) \mathrm{d}s = 0 \quad \text{a.s.}
\]

Thus, $X_t = x_0 + B_t^K$, $t\in [0,T]$, is indeed a solution to equation~\eqref{volterrasde}. It remains to prove that it is the only solution in case the functions $f_i$ in Assumption \ref{ass:A2} are bounded. Suppose that $X_t$, $t\in [0,T]$, solves \eqref{volterrasde} under these conditions, i.e.
$$ X_t = x_0 +  \int_0^t b(s,X_s) \mathrm{d} s+ B_t^K, \quad t \in [0,T].$$
Since $b$ is bounded, from Lemma \ref{lemma:bound-probability-small-interval} and Remark \ref{remark:probzero} we derive that
$$\mathbb{P}(X_s \in M_i) = 0$$
for every $s\in (0,T]$ and $i=1,\ldots, m$. Using this we obtain
\begin{align*}
    \mathbb{E}\left[ \left| \int_0^t b(s,X_s) \mathrm{d}s \right|\right]& \leq \sum_{i=1}^m\mathbb{E}\left[  \int_0^t \left|f_i(X_s)  \right|\mathbbm{1}_{M_i} (X_s)\mathrm{d}s \right] \\
    & \leq c \sum_{i=1}^m  \int_0^t \mathbb{P}\left(X_s \in M_i \right)\mathrm{d}s = 0,
\end{align*}
which proves that 
$$ \int_0^t b(s,X_s) \mathrm{d}s = 0 \quad \text{a.s.}$$
Thus, we have $X_t = x_0 + B_t^K$, $t \in [0,T]$, as claimed.
\end{proof}

With Theorem \ref{thm2:strong_solution} established, we now turn to the proof of Theorem \ref{thm3:strong_solution}, following similar arguments.

\begin{proof}[Proof of Theorem~\ref{thm3:strong_solution}]
As in the proof of Theorem~\ref{thm2:strong_solution}, we explicitly construct the solution, namely it is
\[
    X_t = x_0 + B_t^K + t, \quad t \in [0,T].
\]
We then establish uniqueness by showing that any solution \( X_t \) must necessarily take this form. 

For existence, observe that, similarly to the proof of Theorem~\ref{thm2:strong_solution}, we deduce
\[
    \int_0^t \mathbbm{1}_M (x_0 + B_s^K + s) \mathrm{d}s = \int_0^t 1 \, \mathrm{d}s - \int_0^t \mathbbm{1}_{\mathbb{R} \setminus M} (x_0 + B_s^K + s) \mathrm{d}s = t,
\]
since \( x_0 + B_s^K + s \) is Gaussian with mean \( x_0 + s \) and variance \( \int_0^s K(s,u)^2 \, \mathrm{d}u \), for every $s \in (0,t)$. It follows that this process indeed solves the SDE \eqref{volterrasde}.

Regarding uniqueness, assume that \( (X_t)_{t \in [0,T]} \) is a solution to \eqref{volterrasde}. Then, applying Lemma~\ref{lemma:bound-probability-small-interval} and Remark~\ref{remark:probzero}, we deduce
\[
    \int_0^t \mathbbm{1}_M (X_s) \mathrm{d}s = t \quad \text{a.s.}
\]
Consequently, every solution satisfies
\[
    X_t = x_0 + B_t^K + t, \quad t \in [0,T].
\]
This proves the first assertion of the theorem.

For the second assertion, we argue similarly to show that every solution to \eqref{volterrasde} must satisfy
\[
    \int_0^t b(s,X_s) \mathrm{d}s = \int_0^t \tilde{b} (s,X_s) \mathrm{d}s \quad \text{a.s.}
\]
This completes the proof.
\end{proof}

%\section{Existence and uniqueness of the regularized differential equation with stabilizing term}
%
\section{Convergence of mixed equations towards a weak solution}\label{sec:4weak}

The purpose of this section is to prove that weak solutions to SDEs driven by a single fractional Brownian motion with \( H > 1/2 \) arise as a limit in certain Besov spaces of solutions to SDEs driven by a mixture of completely correlated fractional Brownian motions with two different Hurst parameters. Equations with mixed fBms can be regularized for a larger class of drift functions. More precisely, for given $\sigma \in (0, \infty)$ consider
\begin{align*}
      \mathrm{d}X_t^\sigma & = b(t, X_t^\sigma)\, \mathrm{d}t + \sigma\, \mathrm{d}B_t^{H_1} + \mathrm{d}B_t^{H_2}, \quad t \in [0,T], \\
      X^\sigma_0 &= x_0 \in \mathbb{R},
\end{align*}
with \( H_1 \leq 1/2 \) and \( H_2 = H > 1/2 \). Such an equation is well-posed under a linear growth condition on \( b \), for every \( \sigma > 0 \). The well-posedness result follows, for example, from \cite{buthenhoff} (see also \cite[Theorem 5.1]{nualart2022regularization}).

A key observation is that the regularization effect of the fBm with the smaller Hurst index dominates, regardless of the considered regime. The idea now is to set \( \sigma = 1/N \) for \( N \in \mathbb{N} \) and then let \( N \to \infty \), thus viewing \( \mathrm{d}B_t^{H_1} \) as a stabilizing term and ultimately showing that the equation perturbed solely by \( \mathrm{d}B_t^{H_2} \) also has a solution. Such a procedure, in the case of a single noise,
\begin{align*}
        \mathrm{d}X_t^\varepsilon &= b(t, X_t^\varepsilon)\, \mathrm{d}t + \varepsilon\, \mathrm{d}B_t^H, \quad t \in [0,T], \\
      X_0 &= x_0 \in \mathbb{R},
\end{align*}
has been considered in \cite{pilipenko} and, more recently, in \cite{gassiat}, for a special class of drift functions. This phenomenon has been termed \enquote{selection by noise}, since, as \( \varepsilon \to 0 \), the equation converges to a %"physical" 
solution of the corresponding ODE. The behavior of the equation with two noise terms is crucially different, as will be demonstrated here through various disguises.

In our setting, one would intuitively expect the limit to converge, moreover, to converge to a solution of the equation with solely one noise (without the stabilizing term). Although this expectation is intuitively clear, achieving this convergence involves several technical challenges. From now on, we impose the assumptions that the measurable drift function~$b\colon[0,T]\times\mathbb{R}\to\mathbb{R}$ satisfies:

\begin{itemize}
    \item[(B1)] \label{cond:B1} $b$ is uniformly bounded for $(t, x) \in [0, T] \times \mathbb{R}$, i.e., there exists a $c \in (0,\infty)$ such that for all $(t,x) \in [0,T] \times \mathbb{R}$ we have $|b(t,x)| \le c$.
    \item[(B2)] \label{cond:B2} $b$ is Lipschitz continuous in $x$ (uniformly in $t$) with Lipschitz constant $L$ on a finite number of intervals 
    \begin{align*}
        I_0 = (-\infty,a_1), \, I_1 = (a_1,a_2), \ldots, I_{m-1} = (a_{m-1},a_m), \, I_m = (a_m,+\infty)\,,
    \end{align*}
    where $a_1 < a_2 < \hdots < a_m$ are discontinuity points. That means for all $j = 0,\ldots,m$, for all $t \in [0,T]$ and for all $x,y \in I_j$
    \begin{align*}
        |b(t,x) - b(t,y)| \le L|x-y| \,.
    \end{align*}
\end{itemize}

We will construct weak solutions to 
\begin{equation}
    X_t = x_0 + \int_0^t b(s,X_s) \mathrm{d}s + B_t^{H_2}\,, \quad t \in [0,T], \, x_0 \in \mathbb{R}.
    \label{eq:main_sde}
\end{equation}
Let $N \in \mathbb{N}$ and introduce a stabilizing term which corresponds to a fractional Brownian motion with Hurst parameter $H_1 \in (0,1/2]$. Denote the corresponding sequence of equations with the stabilizing term by
\begin{align}
    X_t^N = x_0 + \int_0^t b(s,X_s^N) \mathrm{d}s + \frac{1}{N}B_t^{H_1} + B_t^{H_2}, \,\quad t \in [0,T], \, x_0 \in \mathbb{R}.
    \label{eq:sequence-sde}
\end{align}

Here \((B_t^{H_1})_{t \in [0,T]}\) and \((B_t^{H_2})_{t \in [0,T]}\) are two completely correlated fractional Brownian motions, meaning that
\[
\frac{1}{N} B_t^{H_1} + B_t^{H_2} = \int_0^t \left(\frac{1}{N} K_{H_1}(t,s) + K_{H_2}(t,s)\right) \, \mathrm{d}W_s, \quad t \in [0,T],
\]
where \(K_{H_1}(t,s)\) and \(K_{H_2}(t,s)\) are the Volterra kernels discussed in Section \ref{sec1.1:leading}. As mentioned, for all $N \in \mathbb{N}$ we have unique strong solutions $(X_t^N)_{t\in [0,T]}$ for the SDE \eqref{eq:sequence-sde} under the condition~\condref{cond:B1}{B1}. The additional assumption \condref{cond:B2}{B2} guarantees that for $N \to \infty$ the sequence of solutions $X^N$ converges to a weak solution for the SDE \eqref{eq:main_sde} in a stochastic Besov space. 
In order to prove this we will use Kolmogorov's criterion for tightness, see \citewithin{Theorem 21.42}{klenke2013probability}\label{thm:kolmogorov-tightness} for example. We recall that the stochastic Besov space with parameter $\beta \in (0,H_1)$ is given by
    \begin{align*}
        W^\beta[0,T] \coloneqq \klg{Y = Y_t(\omega); (t,\omega) \in [0,T] \times \Omega, \norm{Y}_\beta < \infty}
    \end{align*}
    with norm
    \begin{align*}
        \norm{Y}_\beta \coloneqq \sup_{t \in [0,T]}\klr{\mathbb{E}[Y_t^2] + \mathbb{E}\kle{\klr{\int_0^t \frac{|Y_t - Y_s|}{(t-s)^{1+\beta}} \mathrm{d}s}^2}} \,.
    \end{align*}
Let us show that the solutions $X^N$ belong to this stochastic Besov space.
\begin{proposition}\label{prop:X-in-besov}
    Let us assume that the function $b$ in equation \eqref{eq:sequence-sde} satisfies \condref{cond:B1}{B1}. Then the solution of equation \eqref{eq:sequence-sde} belongs to the stochastic Besov space $W^\beta[0,T]$ with norm $\norm{\cdot}_\beta$ for all $N \in \mathbb{N}$ and all $\beta \in (0,H_1)$.
\end{proposition}

\begin{proof}
    The following proof is based on \cite[Theorem 3.2.6]{mishura2008stochastic}. To prove the statement we want to estimate
    \begin{align*}
        A_1^N(t) + A_2^N(t) \coloneqq \mathbb{E}[({X_t^N})^2] + \mathbb{E}\kle{\klr{\int_0^t \frac{|X_t^N - X_s^N|}{(t-s)^{1+\beta}} \mathrm{d}s}^2} \,.
    \end{align*}
    The first term satisfies
    \begin{align*}
        A_1^N(t) &\le 4x_0^2 + 4 \mathbb{E}\kle{\klr{\int_0^t b(s,X_s^N) \mathrm{d}s}^2} + 4\mathbb{E}[(B_t^{H_2})^2] + \frac{4}{N^2} \mathbb{E}[(B_t^{H_1})^2] \\
        &\le 4 \klr{x_0^2 + c^2T^2 + T^{2H_2} + \frac{T^{2H_1}}{N^2}} < \infty \,.
    \end{align*}
    Furthermore, we obtain for the second term
    \begin{align*}
        A_2^N(t) &\le 3 \mathbb{E}\kle{\klr{\int_0^t \frac{|\int_s^t b(u,X_u^N) \mathrm{d}u|}{(t-s)^{1+\beta}} \mathrm{d}s}^2} + 3 \mathbb{E}\kle{\klr{\int_0^t \frac{|B_t^{H_2} - B_s^{H_2}|}{(t-s)^{1+\beta}} \mathrm{d}s}^2} \\
        &\qquad + \frac{3}{N^2} \mathbb{E}\kle{\klr{\int_0^t \frac{|B_t^{H_1} - B_s^{{H_1}}|}{(t-s)^{1+\beta}} \mathrm{d}s}^2} \\
        &\le Cc^2 t^{2-2\beta} + 3 \mathbb{E}\kle{\klr{\int_0^t \frac{|B_t^{H_2} - B_s^{H_2}|}{(t-s)^{1+\beta}} \mathrm{d}s}^2} + \frac{3}{N^2} \mathbb{E}\kle{\klr{\int_0^t \frac{|B_t^{H_1} - B_s^{{H_1}}|}{(t-s)^{1+\beta}} \mathrm{d}s}^2} \,,
    \end{align*}
    where $C > 0$ denotes some constant. Now we can use the self-similarity of the fBm. The second term satisfies
    \begin{align*}
        A_2^N(t) \le Cc^2 t^{2-2\beta} + 3 \mathbb{E}[|B_1^{H_2}|^2] \klr{\int_0^t \frac{|t-s|^{H_2}}{(t-s)^{1+\beta}} \mathrm{d}s}^2 + \frac{3 \mathbb{E}[|B_1^{H_1}|^2]}{N^2} \klr{\int_0^t \frac{|t-s|^{H_1}}{(t-s)^{1+\beta}}\mathrm{d}s}^2 \,.
    \end{align*}
    Since 
    \begin{align*}
        \int_0^t \frac{|t - s|^H}{(t-s)^{1+\beta}} \mathrm{d}s = \frac{t^{H-\beta}}{H - \beta}
    \end{align*}
    for any $t > 0$, $\beta \in (0,H)$ and $H \in (0,1)$, we can conclude that
    \begin{align*}
        A_2^N(t) \le C\klr{c^2 t^{2-2\beta} + t^{2H_1-2\beta} + t^{2H_2-2\beta}} \le C \klr{c^2 T^{2-2\beta} + T^{2H_1-2\beta} + T^{2H_2-2\beta}} < \infty \,,
    \end{align*}
    i.e.\ $\norm{X^N}_\beta < \infty$ and $X^N \in W^\beta[0,T]$ for all $N \in \mathbb{N}$.
\end{proof}

We next prove the following assertion.
%
%%   Let $(X^N)_{N \in \nat}$ be a sequence of continuous stochastic processes on a Polish space. Assume that the following conditions are satisfied
  %  \begin{enumerate}
   %     \item[(i)] The family $(\mathbb{P}(X_0^N \in \cdot))_{N \in \nat}$ of initial distributions is tight.
    %    \item[(ii)] For any $s,t \in [0,T]$ there are numbers $C,\alpha,\beta > 0$ such that for all $N \in \nat$ we have
     %   \begin{align*}
      %      \mathbb{E}\kle{|X_s^N - X_t^N|^\alpha} \le C |s-t|^{\beta + 1} \,.
       % \end{align*}
    %\end{enumerate}
   % Then, the family $(\mathbb{P}_{X^N})_{N \in \nat}$ of distributions of $X^N$ is tight in the set of all probability measures on $C([0,T])$.
%\end{theorem}
%
%\begin{proof}
%   Apply Prohorov’s theorem and follow the proof given in ref.\ \cite[Theorem 21.42]{klenke2013probability}.
%\end{proof}
%
\begin{lemma}[Tightness]\label{lemma:tightness}
    Assume that condition \condref{cond:B1}{B1} is satisfied and denote by $X^N$, $N\in \mathbb{N}$, the sequence of solutions of equation \eqref{eq:sequence-sde}. Then, the family $(X^N,\int_0^\cdot b(s,X^N_s)\mathrm{d}s,B^{H_1}/N,B^{H_2})_{N\in\mathbb{N}}$ is tight in the set of all continuous functions from $[0,T]$ to $\mathbb{R}^4$.
\end{lemma}
\begin{proof}
First, observe that for \(0 < s < t < T\), by assumption \condref{cond:B1}{B1} we have
\[
    |X^N_t - X^N_s| \leq c|t-s| + A(B^{H_1},B^{H_2})\, |t-s|^\gamma,
\]
for every \(\gamma < H_1\), uniformly in \(N\). Here, \(A(B^{H_1},B^{H_2})\) is defined as the maximum of the Hölder constants of the two fractional Brownian motions \((B_t^{H_1})_{t\in [0,T]}\) and \((B_t^{H_2})_{t\in [0,T]}\). In particular, it has moments of arbitrary order.

Next, note that the process
\[
    \int_0^t b(s,X^N_s) \, \mathrm{d} s, \quad t\in [0,T],
\]
is uniformly bounded. Thus, the assertion is a consequence of Kolmogorov's criterion for tightness.
\end{proof}

Having established tightness, we next apply Skorohod's representation theorem to a convergent subsequence and obtain the following result.

\begin{proposition}[Almost sure convergence]\label{prop:as-convergence}
    Under the condition \condref{cond:B1}{B1}, there exists a probability space $(\tilde{\Omega},\tilde{\mathcal{A}},\tilde{\mathbb{P}})$, a family of stochastic processes $(\tilde{X}^N,\int_0^\cdot {b}(s,\tilde{X}_s^N) \mathrm{d}s, \tilde{B}^{H_1}/N, \tilde{B}^{H_2})_{N\in\mathbb{N}}$ and a subsequence $(N_k)_{k \in \mathbb{N}}$ such that
    \begin{align*}
        (X^{N_k},\int_0^\cdot {b}(s,{X}_s^{N_k}) \mathrm{d}s,B^{H_1}/N_k,B^{H_2}) \overset{d}{=} (\tilde{X}^{N_k},\int_0^\cdot {b}(s,\tilde{X}_s^{N_k})\mathrm{d}s,\tilde{B}^{H_1}/N_k,\tilde{B}^{H_2})
    \end{align*}
    for all $k \in \mathbb{N}$. Moreover, the family converges almost surely with respect to the subsequence and satisfies the SDE
    \begin{align*}
        \tilde{X}_t^{N_k} = x_0 + \int_0^t {b}(s,\tilde{X}_s^{N_k}) \mathrm{d}s + \frac{1}{N_k}\tilde{B}_t^{H_1} + \tilde{B}_t^{H_2}, \quad t \in [0,T], 
    \end{align*}
    for all $k \in \mathbb{N}$.
\end{proposition}
\begin{proof}
    Due to Lemma \ref{lemma:tightness} we know that the family of stochastic processes is tight. Hence, there exists a subsequence of the family $(X^{N_k},\int_0^\cdot {b}(s,{X}_s^{N_k}) \mathrm{d}s,B^{H_1}/N_k,B^{H_2})_{k \in \nat}$ 
    that converges weakly. Therefore, we can apply Skorohod's representation theorem to conclude that there exist a probability space $(\tilde{\Omega},\tilde{\mathcal{A}},\tilde{\mathbb{P}})$ and processes $(\tilde{X}^{N_k},\int_0^\cdot {b}(s,\tilde{X}_s^{N_k})\mathrm{d}s,\tilde{B}^{H_1}/N_k,\tilde{B}^{H_2})_{k \in \nat}$ which converge almost surely such that 
    \begin{align*}
        (X^{N_k},\int_0^\cdot {b}(s,{X}_s^{N_k}) \mathrm{d}s,B^{H_1}/N_k,B^{H_2}) \overset{d}{=} (\tilde{X}^{N_k},\int_0^\cdot {b}(s,\tilde{X}_s^{N_k})\mathrm{d}s,\tilde{B}^{H_1}/N_k,\tilde{B}^{H_2}) .
    \end{align*}
    Since they coincide in law, the family needs to satisfy the SDE 
    \begin{align*}
        \tilde{X}_t^{N_k} = x_0 + \int_0^t {b}(s,\tilde{X}_s^{N_k}) \mathrm{d}s + \frac{1}{N_k}\tilde{B}_t^{H_1} + \tilde{B}_t^{H_2} , \quad t \in [0,T], 
    \end{align*}
    with probability one for all $k \in \nat$.
\end{proof}

From now on, by abuse of notation, we will not distinguish between the family 
$$( X^N, \int_0^\cdot b(s,X^N_s)\,\mathrm{d}s, \frac{B^{H_1}}{N}, B^{H_2} )_{N\in \mathbb{N}}$$
and the version constructed in Proposition \ref{prop:as-convergence}. Fix $t\in (0,T]$ and let $\varepsilon \in (0,t)$. Recall the CGP constructed in equation \eqref{eq:definition-Y} given here by
\begin{align*}
    Y^N_t(\varepsilon) \coloneqq X_{t-\varepsilon}^N + \frac{1}{N}\klr{B_t^{H_1} - B_{t-\varepsilon}^{H_1}} + \klr{B_t^{H_2} - B_{t-\varepsilon}^{H_2}} \,.
\end{align*}
In Lemma \ref{lemma:Y-conditionally-gaussian} we have shown that given $\mathcal{A}_{t-\varepsilon}$, $Y_t^N(\varepsilon)$ is conditionally Gaussian with mean
\begin{align*}
    \xi^N \coloneqq X^N_{t-\varepsilon} + \int_0^{t-\varepsilon} \klr{ \frac{1}{N} \klr{K_{H_1}(t,s) - K_{H_1}(t-\varepsilon,s)} + \klr{K_{H_2}(t,s) - K_{H_2}(t-\varepsilon,s)}}\mathrm{d}W_s 
\end{align*}
and variance
\begin{align}
    (\kappa^N)^2 \coloneqq \int_{t-\varepsilon}^t \klr{\frac{1}{N} K_{H_1}(t,s) + K_{H_2}(t,s)}^2 \mathrm{d}s \,.
    \label{eq:variance-kappa-m}
\end{align}
Furthermore, recall that we can express $X_t^N$ via $Y_t^N(\varepsilon)$ through
\begin{align*}
    X_t^N &= X_{t-\varepsilon}^N + \int_{t-\varepsilon}^t b(s,X_s^N) \mathrm{d}s + \frac{1}{N}\klr{B_t^{H_1} - B_{t-\varepsilon}^{H_1}} + \klr{B_t^{H_2} - B_{t-\varepsilon}^{H_2}} \nonumber \\
    &= Y^N_t(\varepsilon) + \int_{t-\varepsilon}^t b(s,X_s^N) \mathrm{d}s \,.
\end{align*} 
An important observation is that the variance $(\kappa^N)^2$ in \eqref{eq:variance-kappa-m} satisfies
\begin{align*}
    (\kappa^N)^2 \geq \int_{t-\varepsilon}^t \klr{ K_{H_2}(t,s)}^2 \mathrm{d}s \geq C \int_{t-\varepsilon}^t (t-s)^{2H_2 - 1}\mathrm{d}s \geq C \varepsilon^{2H_2}
\end{align*}
for some constant $C\in (0, \infty)$ independent of $N$, where we used the fact that $H_2>1/2$ and \eqref{fbmkernelbound}. Moreover, assumption \eqref{eq:kernel_assumption} is fulfilled and Lemma \ref{lemma:bound-probability-small-interval} applies. This observation will be used in the following proofs.
%
%Let $M \in \nat$ denote another arbitrary natural number. In Remark \ref{remark:YY-cond-gaussian} we have shown that the process $(Y_t^M, Y_t^N)_{t \in [0,T]}$ is also conditionally Gaussian with mean $(\xi^M,\xi^N)$ and covariance matrix
%
%\begin{align}
%    \Sigma^{M,N} = \matr{(\kappa^M)^2 & \lambda^{M,N}\\ \lambda^{M,N} & (\kappa^N)^2}\,,
%\end{align}
%
%whereas $\lambda^{M,N}$ is defined as the covariance of $Y_t^M(\varepsilon)$ and $Y_t^N(\varepsilon)$, i.e.,
%
%\begin{align*}
%    \lambda^{M,N} = \int_{t-\varepsilon}^t \klr{\frac{\klr{K_{H_1}(t,s)}^2}{MN} + \klr{\frac{1}%{N} + \frac{1}{M}} K_{H_1}(t,s) K_{H_2}(t,s) + \klr{K_{H_2}(t,s)}^2} \mathrm{d}s \,%.
%\end{align*}
%

%

We now prove that the family above converges towards a weak solution of \eqref{eq:main_sde}, first establishing convergence in \( L^2 \), followed by convergence in the stochastic Besov space.  

\begin{theorem}[Existence of weak solution]\label{thm:existence_main}
    Let us assume that conditions \condref{cond:B1}{B1} and \condref{cond:B2}{B2} are satisfied. Define the stochastic process $X'$ as
    \begin{align}
        X'_t \coloneqq x_0 + \int_0^t b(s,X_s) \mathrm{d}s + B_t^{H_2}\,, \quad t\in [0,T],
        \label{eq:X-prime-sde}
    \end{align}
    where $X$ is defined as the unique limit derived in Proposition \ref{prop:as-convergence} of the subsequence $(X^{N_l})_{l \in \nat}$ of solutions of the SDE \eqref{eq:sequence-sde} on some probability space $(\tilde{\Omega},\tilde{\mathcal{A}},\tilde{\mathbb{P}})$. Then, $X = X'$ almost surely on the probability space $(\tilde{\Omega},\tilde{\mathcal{A}},\tilde{\mathbb{P}})$. Moreover, it holds
    $$X^{N_l} \xrightarrow{L^2} X',$$
    that is $(X^{N_l})_{l \in \nat}$ converges towards $X$ in $L^2$.
\end{theorem}
\begin{proof}
    In the following, we can assume without loss of generality that the drift function $b$ has only one discontinuity point in $a_1 = 0$. By abuse of notation, we denote $\tilde{\mathbb{P}}=\mathbb{P}$ and accordingly don't distinguish notation between expected values. We need to show that
    \begin{align*}
        \lim_{l \to \infty} \mathbb{E}\kle{\klr{X_t^{N_l} - X_t'}^2} = 0 
    \end{align*}
    for any $t \in [0,T]$. Let
    \begin{align*}
        A_1^{N_l}(t) \coloneqq \mathbb{E}\kle{\klr{X_t^{N_l} - X_t'}^2} \,.
    \end{align*}
    Since $X^{N_l}$ and $X'$ fulfill the equations \eqref{eq:sequence-sde} and \eqref{eq:X-prime-sde}, we can conclude that $A_1^{N_l}(t) \le 2\klr{I_1 + I_2}$ with
    \begin{align*}
        I_1 &= \mathbb{E}\kle{\klr{\int_0^t \klr{b(s,X_s^{N_l}) - b(s,X_s)} \mathrm{d}s}^2}\,, \\
        I_2 &= \frac{1}{N_l^2} \mathbb{E}\kle{(B_t^{H_1})^2}  = \frac{1}{N_l^2} |t|^{2H_1} \,.
    \end{align*}
    Therefore, only the discussion of $I_1$ is left. If $X_s^{N_l},X_s > 0$ or $X_s^{N_l},X_s < 0$, we can make use of the piecewise Lipschitz continuity and argue analogously to the proof of \cite[Theorem 3.2.8]{mishura2008stochastic}. On the other hand, if $X_s^{N_l} > 0$ and $X_s < 0$ (or the other way around), let $A_s \coloneqq \{X_s^{N_l} > 0; X_s < 0\}$ and $f(s,X_s^{N_l},X_s) = b(s,X_s^{N_l}) - b(s,X_s)$. Note that, since $b$ is bounded, $f$ is also bounded. Then, due to the Cauchy-Schwarz inequality we can estimate $I_1$ via
    \begin{align*}
        I_1 \le C\, \mathbb{E}\kle{\int_0^t \mathbbm{1}_{A_s} \mathrm{d}s} + C\, \mathbb{E}\kle{\int_0^t f(s,X_s^{N_l},X_s)^2 \mathbbm{1}_{A_s^c} \mathrm{d}s} \,.
    \end{align*}
    On $A_s^c$ we have, because of \condref{cond:B2}{B2}, $$f(s,X_s^{N_l},X_s)^2 \le L \klr{X_s^{N_l} - X_s}^2.$$ 
    Therefore, we get
    \begin{align*}
        I_1 \le C \int_0^t \mathbb{E}\kle{\mathbbm{1}_{A_s}} \mathrm{d}s + LC \int_0^t \mathbb{E}\kle{\klr{X_s^{N_l} - X_s}^2} \mathrm{d}s \,.
    \end{align*}
    Now, by definition we have
    \begin{align*}
        \mathbb{E}\kle{\mathbbm{1}_{A_s}} &= \mathbb{P}(X_s^{N_l} > 0, X_s < 0) \nonumber \\
        &\le \mathbb{P}(|X_s^{N_l} - X_s| \ge 1/n, X_s^{N_l} > 0, X_s < 0) + \mathbb{P}(- 1/n < X_s < 0, 0 < X_s^{N_l} < + 1/n)  \\
        &\le \mathbb{P}(|X_s^{N_l} - X_s| \ge 1/n) + \mathbb{P}(- 1/n < X_s < 0, 0 < X_s^{N_l} < + 1/n)
    \end{align*}
    for any $n \in \nat$. We estimate the second term as
    \begin{align*}
        \int_0^t\mathbb{P}(- 1/n < X_s < 0, 0 < X_s^{N_l} < + 1/n) \mathrm{d}s \le \int_0^t \mathbb{P}(X_s^{N_l} \in (0, 1/n)) \mathrm{d}s 
    \end{align*}
    and apply Lemma \ref{lemma:bound-probability-small-interval} to see that
    \begin{align*}
        \int_0^t\mathbb{P}(- 1/n < X_s < 0, 0 < X_s^{N_l} < + 1/n) \mathrm{d}s \le K_1 n^{H_2 - 1}
    \end{align*}
    for some $K_1 > 0$ independent of $n$. In total we have
    \begin{align*}
        A_1^{{N_l}}(t) \le \frac{K_2}{{N_l^2}} |t|^{2H_1} + K_2 n^{H_2-1} + K_2 \int_0^t \mathbb{P}(|X_s^{N_l} - X_s| \ge 1/n) \mathrm{d}s + K_2 \int_0^t A_1^{{N_l}}(s) \mathrm{d}s\,,
    \end{align*}
    where $K_2 > 0$ is a constant independent of $n,{N_l}$. Now we can apply Gronwall's lemma \cite{gronwall1919note} to get
    \begin{align*}
        A_1^{{N_l}}(t) \le K_3\kle{\frac{1}{{N_l^2}} |t|^{2H_1} + n^{H_2-1} + \int_0^t \mathbb{P}(|X_s^{N_l} - X_s| \ge 1/n) \mathrm{d}s} \,.
    \end{align*}
    If we let $l \to \infty$ we get by Proposition \ref{prop:as-convergence} and due to the definition of $X_s$ that
    \begin{align*}
        \limsup_{l\to \infty }A_1^{N_l}(t)\leq K_3 n^{H_2-1}  \,.
    \end{align*}
    Since this holds for every sufficiently large $n\in \nat$, letting $n \to \infty$ we get
    \begin{align*}
        A_1^{N_l}(t) \to 0 \,,
    \end{align*}
    i.e., $X^{N_l} \xrightarrow{L^2} X'$. Moreover, due to $X^{N_l} \xrightarrow{\text{a.s.}} X$, we have $X_t = X_t'$ for all $t \in [0,T]$, with probability one. This completes the proof. 
\end{proof}

Finally, we give a proof for the convergence in the stochastic Besov space.

\begin{proposition}[Convergence in stochastic Besov space]\label{prop:convergence-besov}
    Suppose that the same conditions as in Theorem \ref{thm:existence_main} are satisfied. Then the sequence $(X^{N_l})_{l\in \nat}$ converges towards $X$ with respect to the norm $\norm{\cdot}_\beta$ of the space $W^\beta[0,T]$ for any $\beta \in (0,H_1)$.
\end{proposition}

\begin{proof}
     We can assume again without loss of generality that the drift function $b$ has only one discontinuity point in $a_1 = 0$. We need to show that
    \begin{align*}
        \lim_{l \to \infty} \norm{X_t^{N_l} - X_t}_\beta = 0 \,
    \end{align*}
    for any $t \in [0,T]$. Let us consider
    \begin{align*}
        A_1^{N_l}(t) + A_2^{N_l}(t) \coloneqq \mathbb{E}\kle{\klr{X_t^{N_l} - X_t}^2} + \mathbb{E}\kle{\klr{\int_0^t \frac{\kls{X_t^{N_l} - X_t - X_s^{N_l} + X_s}}{(t-s)^{1+\beta}} \mathrm{d}s}^2} \,.
    \end{align*}
    We have already shown in Theorem \ref{thm:existence_main} that $A_1^{N_l}(t)$ converges to zero for $l \to \infty$, i.e., the discussion of the second term is left. Similar to the proof of \cite[Theorem 3.2.7]{mishura2008stochastic}, we have $A_2^{N_l}(t) \le C\klr{I_3 + I_4}$, where
    \begin{align*}
         I_3 &= \mathbb{E}\kle{\klr{\int_0^t \frac{\int_s^t \klr{b(u,X_u^{N_l}) - b(u,X_u)} \mathrm{d}u}{(t-s)^{1+\beta}} \mathrm{d}s}^2}\,, \\
         I_4 &= \klr{\frac{1}{N_l}}^2 \mathbb{E}\kle{\klr{\int_0^t \frac{|B_t^{H_1} - B_s^{H_1}|}{(t-s)^{1+\beta}} \mathrm{d}s}^2}\,.
    \end{align*}
    Analogously to the discussion of integral $I_1$ above, we can estimate $I_3$ by
    \begin{align*}
        I_3 &\le D \mathbb{E}\kle{\klr{\int_0^t \frac{\int_s^t \mathbbm{1}_{A_u} \mathrm{d}u}{(t-s)^{1+\beta}}\mathrm{d}s}^2} + D \mathbb{E}\kle{\klr{\int_0^t \frac{\int_s^t |{X_u^{N_l} - X_u}| \mathrm{d}u}{(t-s)^{1+\beta}} \mathrm{d}s}^2} \,.
    \end{align*}
    To further estimate $I_3$ and $I_4$, let
    \begin{align*}
        I_5\kle{h(\cdot)} \coloneqq \int_0^t \int_s^t (t-s)^{-1-\beta} h(u) \mathrm{d}u \mathrm{d}s \,.
    \end{align*}
    Then, using Fubini, we can estimate it as 
    \begin{align*}
        I_5[h(\cdot)] \le \int_0^t (t-u)^{-\beta} h(u) \mathrm{d}u \,.
    \end{align*}
    Now we can apply Cauchy-Schwarz and the Markov inequality and argue similarly as for the integral $I_1$ defined in the proof of Theorem \ref{thm:existence_main} to estimate $I_3$ as
    \begin{align*}
        I_3 &\le D_1 n^2 \int_0^t (t-s)^{-2\beta} A_1^{N_l}(s) \mathrm{d}s + D_2 n^{H_2-1} \int_0^t (t-s)^{-2\beta} \mathrm{d}s \,.
    \end{align*}
    The term $I_4$ can be estimated as in the proof of Proposition \ref{prop:X-in-besov} by using the self-similarity of the fBm. Therefore, in total we have 
    \begin{align*}
        A_2^{N_l}(t) &\le B_1 \klr{\frac{1}{N_l}}^2 + B_2 n^{H_2-1} \int_0^t (t-s)^{-2\beta} \mathrm{d}s + B_3 n^2 \int_0^t (t-s)^{-2\beta} A_1^{N_l}(s) \mathrm{d}s \,.
    \end{align*}
    Let first $l \to \infty$ and then $n \to \infty$.
    In particular, for fixed $n$ we have by dominated convergence
    \begin{align*}
        \int_0^t (t-s)^{-2\beta} A_1^{N_l}(s)\mathrm{d}s \to 0
    \end{align*}
    as $l \to \infty$.
    %As before we can use Gronwall's Lemma, let first $l \to \infty$ and then $n \to \infty$. 
    Hence, we obtain the limit $A_2^{N_l}(t) \to 0$, uniformly in $t$, i.e., the sequence $X^{N_l}$ is also convergent towards $X$ in the Besov space.
\end{proof}

%%%%%%%%%%%%%%%%%%%%%%%%%%%%%%%%%%%%%%%%%%%%%%
%% Single Appendix:                         %%
%%%%%%%%%%%%%%%%%%%%%%%%%%%%%%%%%%%%%%%%%%%%%%
%\begin{appendix}
%\section*{???}%% if no title is needed, leave empty \section*{}.
%\end{appendix}
%%%%%%%%%%%%%%%%%%%%%%%%%%%%%%%%%%%%%%%%%%%%%%
%% Multiple Appendixes:                     %%
%%%%%%%%%%%%%%%%%%%%%%%%%%%%%%%%%%%%%%%%%%%%%%
%\begin{appendix}
%\section{???}
%
%\section{???}
%
%\end{appendix}

%%%%%%%%%%%%%%%%%%%%%%%%%%%%%%%%%%%%%%%%%%%%%%
%% Support information, if any,             %%
%% should be provided in the                %%
%% Acknowledgements section.                %%
%%%%%%%%%%%%%%%%%%%%%%%%%%%%%%%%%%%%%%%%%%%%%%
\begin{acks}[Acknowledgments]
    The authors would like to thank the referees for careful reading of the manuscript and for helpful comments and suggestions that improved the paper.
\end{acks}
\bibliographystyle{imsart-nameyear} % Style BST file (imsart-number.bst or imsart-nameyear.bst)
\bibliography{main}       % Bibliography file (usually '*.bib')

%% or include bibliography directly:
% \begin{thebibliography}{}
% \bibitem[\protect\citeauthoryear{???}{???}]{b1}
% \end{thebibliography}

\end{document}